\documentclass[a4paper,10pt]{article}

\usepackage{amssymb}
\usepackage{amsmath}
\usepackage{amscd}
\usepackage{amsthm}
\usepackage{graphicx}
\usepackage[all]{xy}

\usepackage{color}          
\usepackage{epsfig}

\usepackage{lineno}
\newcommand{\B}{\mathbb B}

\newcommand{\C} {\mathbb C}

\newcommand{\h} {\hat}
\newcommand{\mc}{\mathcal}

\newtheorem{theorem}{Theorem}
\newtheorem{lemma}[theorem]{Lemma}
\newtheorem{corollary}[theorem]{Corollary}
\newtheorem{proposition}[theorem]{Proposition}

\title{On Poincar\'e extensions of rational maps.}
\author{Carlos Cabrera, Peter Makienko and Guillermo Sienra.}

\begin{document}

\maketitle
\footnotetext{This work was partially supported by PAPIIT IN-105912 
and CONACYT CB2010/153850.}
\begin{abstract}

There is a classical extension, of M\"obius automorphisms of the Riemann sphere 
into isometries of the hyperbolic space $\mathbb{H}^3$, which is called the 
Poincar\'e extension. In this paper, we construct extensions of rational maps 
on the Riemann sphere over endomorphisms of $\mathbb{H}^3$ exploiting the 
fact that any holomorphic covering between Riemann surfaces is M\"obius for a 
suitable choice of coordinates. 
We show that these extensions define  conformally natural homomorphisms on 
suitable subsemigroups of the semigroup of Blaschke maps.  We extend the 
complex multiplication to a product in $\mathbb{H}^3$ that allows to construct 
a visual extension 
of any given rational map.
\end{abstract}

\section{Introduction}

In the literature there are some constructions of extensions of rational
dynamics from $\C$ to $\mathbb{H}^3$, see for example \cite{martinextension}, 
\cite{Mc2} and \cite{Carsten}. The constructions in \cite{Mc2} and 
\cite{Carsten} are based on Choquet's barycentric construction introduced and 
studied by A. Douady and C. Earle in their paper \cite{DouadEarl}. Other 
important contributions on the barycentric constructions appear in 
\cite{Abikofbarycen} and \cite{bessonCourGallot}.

As it was mention in the abstract, the basic idea of  this paper is the 
following fact: ``\textit{Any holomorphic covering between Riemann surfaces is 
a M\"obius map on suitable coordinates.}'' Then this covering can be extended 
to suitable M\"obius manifolds. Let us discuss this idea in details. 

First,  remind that a M\"obius $n$-orbifold is a $n$-orbifold endowed with an
atlas such that the transition maps are M\"obius transformations.

Given a discrete subgroup $\Gamma$ of M\"{o}bius transformations of the $n$-sphere $S^{n}$
acting properly discontinuous and freely on a domain $\Omega \subset S^{n}$, the
quotient manifold $\Omega / \Gamma$ admits a M\"{o}bius structure. In the case
when $n=3$, any manifold modeled on one of the following spaces 
$\mathbb{R}^{3}, S^{3}$, the unit ball $B^{3}$ in $\mathbb{R}^3, \, S^{2}\times 
{\mathbb 
{R}}$ or $B^{2}\times
{\mathbb{R}}$ admits a M\"{o}bius structure, see \cite{ScGeom}.

Let $S_1$ and $S_2$ be two M\"obius  2-orbifolds and let
$R:S_1\rightarrow S_2$ be a finite degree covering which is M\"obius on the
respective M\"obius structures. Assume that there exist two Kleinian groups
$\Gamma_1$ and $\Gamma_2$ and two components  $W_1$ and $W_2$ of the 
discontinuity
sets $\Omega(\Gamma_1)$ and $\Omega(\Gamma_2)$ respectively, such that
$$S_i=W_i/Stab_{W_i}(\Gamma_i)$$ for $i=1,2.$ Now assume that there exist a
 M\"obius map $\alpha(R):W_1\rightarrow W_2$ making the following diagram
commutative

\begin{displaymath}
    \xymatrix{W_1
         \ar[d]_{\pi_1} \ar[r]^{\alpha(R)} & 
\ar[d]^{\pi_2}W_2\\S_1
           \ar[r]^{R}  & S_2
}\tag{1}\label{diag.basic}
\end{displaymath}
so that $\alpha(R)$ induces a homomorphism from $\Gamma_1$ to
$\Gamma_2.$ If $$M_i=(\bar{\mathbb{H}}^3\cup W_i)/\Gamma_i,$$ then $\alpha(R)$
induces a unique M\"obius morphism $$\tilde{R}:M_1\rightarrow M_2$$ which is an
extension of $R:S_1\rightarrow S_2.$ We call the map $\tilde{R}:M_1\rightarrow M_2$ a
\textit{Poincar\'e extension} of $R$. The map $\tilde{R}$ depends on the uniformizing groups $\Gamma_1$ and
$\Gamma_2$. Hence, in general,  for a given covering map $R$ there are many
possibilities to construct a Poincar\'e extension. Note that
the degree $deg(\tilde{R})$ is equal to the index $[\Gamma_2:\alpha(R)\circ\Gamma_1\circ
\alpha(R)^{-1}]. $
Hence $deg(R)\leq deg(\tilde{R})$ with equality when
$$Stab_{W_i}(\Gamma_i)=\Gamma_i.$$

Given a Riemann surface $S$ with a fixed M\"obius structure, in 
\cite[sect.8]{KulkarniPinkal} R. Kulkarni and U. Pinkal constructed 
 a M\"obius 3-manifold $M$ such that  the surface $S$ is canonically contained 
in  the
boundary of $M$. If the structure of $S$ is uniformazible by a non trivial
Kleinian group then, this construction is given by the Classical Poincar\'e extension of
the uniformizing group and produces a complete hyperbolic manifold $M$. 
The construction is based on the following idea: Let $D$ be a round disk on
$S$ with respect to the M\"obius structure; that is, there exist a coordinate 
under which $D$ is a round disk in the plane. Using this coordinate, we attach 
a round half-ball in $\mathbb{H}^3$ to $D$. Then the 3-manifold  $M$ is the 
union of all the half open balls over all round disks in $S.$ 

On the Riemann sphere $\bar{\C}$ there is a unique complete M\"obius structure $\sigma_0$, this is
the standard M\"obius structure on $\bar{\C}$. The construction of Kulkarni 
and Pinkal is clearer when $S$ is a planar surface with the standard 
M\"obius structure. Let $R$ be a branched self-covering of $\bar{\C}$. If 
$deg(R)>1$, then $R$ is not a M\"obius 
covering with respect to the standard structure on any domain in
$\bar{\C}.$

When $S=\C^*$ with the standard M\"obius structure, Kulkarni-Pinkal
construction gives a canonical non-complete M\"obius 3-manifold which
is M\"obius equivalent to the $3$-dimensional ball with the vertical diameter
removed and endowed with the standard conformal structure on $B^3$. Now, 
consider a 
complete M\"obius structure on $\C^*$. In this case, Kulkarni-Pinkal 
construction gives a complete hyperbolic 3-manifold with the same underlying 
space as before. More generally, if $S=\bar{\C}\setminus F$, where $F$ is a closed set, 
then the Kulkarni-Pinkal's extension $M$ is homeomorphic
to $\mathbb{H}^3\setminus convhull(F)$, where $convhull(F)$ is the
hyperbolic convex hull  in $\mathbb{ H}^3$ of all points in $F$. The standard
M\"obius structure on $M$ is the extension of the standard M\"obius structure on
$S.$ The construction of Kulkarni and Pinkal motivates the idea of a model
manifold for the Poincar\'e extension of a rational map. 

We will restrict our attention to the case when $R$  is a rational map and
$S_1$ and $S_2$ are two Riemannian orbifolds with underlying spaces 
contained in $\bar{\C}$, and such that
$R:S_1\rightarrow S_2$ is a holomorphic  covering. Let $\sigma_2$ be
an uniformizable M\"obius structure on
$S_2$ and suppose that the pullback $\sigma_1:=R_*(\sigma_2)$ is also a 
uniformizable M\"obius
structure on $S_1$. If $\Gamma_1$ and $\Gamma_2$ are the uniformizing groups.
Let $\tilde{R}$ be a Poincar\'e extension of $R$, such that $\Omega(\Gamma_i)$
are connected. Let $\phi_i:\partial M_i\rightarrow S_i$ be some identification maps and 
assume that there are  homeomorphic extensions $\Phi_i:M_i\rightarrow
\bar{\mathbb{H}}^3$ for each $\phi_i$. Then the map $\Phi_2  \circ \tilde{R} \circ \Phi_1^{-1}$ is
called \textit{geometric extension} if and only if satisfies the following
conditions.
\begin{enumerate}

\item The sets $\Phi_i(M_i\cup \partial M_i)$  are of the form
$\bar{\mathbb{H}}^3\setminus
\{\bigcup \gamma_j\}$
where each $\gamma_j$ is either a quasi-geodesic or  a family of finitely many
quasi-geodesic rays with common starting point. There are  no more than 
countably many curves $\gamma_j$.  Here by quasi-geodesic we mean the image of 
a hyperbolic geodesic by a quasiconformal automorphism.

\item There exist a continuous extension,  on all $\mathbb{H}^3$, which
maps
complementary quasi-geodesics to complementary quasi-geodesics. 

\end{enumerate}
Hence, a geometric extension is an endomorphism
of $\mathbb{H}^3$ such that 
its restriction to $\Phi_1(M_1)$ is a Poincar\'e extension.

Let $Rat_d(\C)$ denote the set of rational maps $R$ of degree $d$.
Let $A\subset Rat_d(\C)$. Assume that there exist a map
$$Ext:A\rightarrow End(\bar{\mathbb{H}}^3)$$ such that
$Ext(R)$ is an extension of $R$ for every $R$ in $A$. Then for every
pair of maps $h,g$ in the M\"obius group $Mob$ we define
$$\widetilde{Ext}(g\circ R \circ h)=\hat{g}\circ Ext(R) \circ \hat{h}$$
where $\hat{g}$ and $\hat{h}$ are the classical Poincar\'e extensions
of the maps $g$ and $h$ in the hyperbolic space, respectively.

If $\widetilde{Ext}$ is a well defined map from the  M\"obius bi-orbit of $A$ to
$End(\bar{\mathbb{H}}^3)$, then we call $Ext$ a \textit{conformally natural
extension} of $A$.

In particular, when the bi-action of $PSL(2,\C)$ on $A$
has no fixed points on $A$, then any map $Ext:A\rightarrow 
End(\mathbb{H}^3)$ defines a map $\widetilde{Ext}$ on the M\"obius bi-orbit 
which is a conformal natural extension. If the action has fixed points then, in 
order to obtain a conformally
natural extension the map $Ext$ has to be consistent with the M\"obius action.
The situation is tricky, even in the case when $A$ consists of a single point
$R$.

Let $\hat{R}$ be a geometric extension of a rational map $R$, then any
rational map on the M\"obius bi-orbit of $R$ has a geometric extension.
Namely, if $h$ and $g$ are elements in $PSL(2,\C)$, then $Q=g\circ R \circ
h$ has a geometric extension with the same uniformizing groups $\Gamma_1$ and 
$\Gamma_2$, the projections $p_1=\pi_1\circ \hat{h}$ and $p_2= 
\pi_2\circ\hat{g}^{-1}$ and the associated manifolds are 
$N_1=\hat{h}^{-1}(M_1)$ and $N_2=\hat{g}(M_2).$
We define an extension of $Q$  by the formula $\hat{Q}=\hat{g}\circ \hat{R}
\circ \hat{h}.$ 

However, one should be careful in the situation when, for a given $R$, there are other elements
$g'$ and  $h'$ in $PSL(2,\C)$ such that $Q=g'\circ R \circ h'$. This situation happens when there 
are elements $h_1$ and $h_2$ in $PSL(2,\C)$  such that

\begin{displaymath} R\circ h_2=h_1\circ
R.\tag{*}\label{Mob.semiconj}\end{displaymath} 

If there are no such $h_1$ and $h_2$ then, any geometric extension of $R$ is
conformally natural.  However,  if such elements do exist  but the Poincar\'e 
extensions of $h_i$ are M\"obius automorphisms of $M_i$ with respect to their M\"obius structures,
then the extension $R\mapsto \hat{R}$ is  conformally natural.

In this article, we will investigate the existence of extensions of $R$,
defined in the hyperbolic space $\mathbb{H}^3$,  satisfying as many 
as possible of the following desirable conditions.

\begin{enumerate}
\item \textbf{Geometric.} As defined above.

 \item \textbf{Same degree.} The index $[\Gamma_2:\Gamma_1]$ is equal
to the degree of the map $R$.

\item \textbf{Dynamical.}   Let $\hat{R}$ be a Poincar\'e extension
such that $M_1=M_2$, then $\hat{R}$ is \textit{dynamical}. In particular, these
are extensions $Ext$ such that  $Ext(R^n)=Ext(R)^n$ for $n=1,2,....$

\item \textbf{Semigroup Homomorphisms.} A stronger version of  the previous
property is to find semigroups $\mathcal{S}$, of rational maps, for which there
is an extension $Ext$ defined in all $\mathcal{S}$ such that $Ext(R\circ
Q)=Ext(R)\circ Ext(Q)$.

\item \textbf{Equivariance under M\"obius actions.} We look for conformally
natural extensions of subsets $A$ of $Rat_d(\C).$
\end{enumerate}

Assume that $R$ has a geometric extension, and let $\Gamma_2$ be the group that
uniformizes the  M\"obius structure on $S_2$. Then the discontinuity set 
$\Omega(\Gamma_2)$  consists 
of the orbit of a unique component $C$. The stabilizer of $C$ uniformizes 
the surface $S_2$. In this case the group $\Gamma_2$ is either
totally degenerated, of Schottky type or of Web group type. When $\Gamma_2$ is 
a web group, the orbit of the component $C$ is infinite. Thus in general it is 
possible 
that the condition (2) may not be satisfied.

On the other hand, if the component $C$ uniformizing  the surface $S_2$ 
is invariant under  $\Gamma_2$, then condition (2) is satisfied. For this 
reason, we restrict our discussion to this case. A group having an invariant 
component is called a function group. By Maskit's theorem any function group 
can be represented as a Klein-Maskit combination of the following groups:

\begin{itemize}
 \item Totally degenerated groups.
\item Schottky type.

\end{itemize}
According to this list of groups we call a geometric extension 
\textit{totally degenerated}, or of \textit{Schottky type}, whenever the 
uniformizing
group has the corresponding property.
Totally degenerated groups appear as geometric limits of quasifuchsian groups.
In fact, totally degenerated groups belong to the boundary of a Bers' slice. 

This paper is organized as follows.

In Section 2 we will discuss Hurwitz spaces and quasifuchsian extensions.
We will construct an extension, the radial extension, that satisfies (2), (3)
and (4) in the previous list. We will give some conditions for which the radial
extension is geometric.

Section 3 is devoted to Schottky type extensions.

In Section 4, we discuss a visual extension of all rational maps which is
connected to a product structure defined on the hyperbolic space.

Finally, in Section 5, we discuss examples of extensions and surgeries of Maskit
type.

\par\medskip\noindent \textbf{Acknowledgment.}
The authors would like to thank M. Kapovich for useful comments and suggestions 
on an early draft of this paper.

\section{Fuchsian structures, degenerated and radial extensions}

Given a rational map $R$, let $CV(R)$ denote the critical values of $R$ and 
take $S_2=\bar{\mathbb{C}}\setminus CV(R)$ and $S_1=R^{-1}(S_2)$ then 
$R:S_1\rightarrow S_2$ is a covering.  We assume that
the set of critical values of $R$ contains at least three points, say $b_1, b_2$
and $b_3$, so that $S_2$ is a hyperbolic Riemann surface. In order to get normalized maps
we pick three points $a_1, a_2$ and $a_3$ in $S_1$ such that $R(a_i)=b_i$ for $i=1,2,3.$

We say that two branched coverings $R$ and $Q$,
of the Riemann sphere onto itself, are \textit{Hurwitz equivalent} if there are 
quasiconformal homeomorphisms $\phi$ and $\psi$,  making the following diagram
commutative
\begin{displaymath}
    \xymatrix{
        \hat{\C} \ar[d]_{R} \ar[r]^{\psi} &
\ar[d]^{Q} \hat{\C} \\
             \hat{\C}  \ar[r]^{\phi}  &   \hat{\C}} 
\end{displaymath}

Given a rational map $R$, the Hurwitz space $H(R)$ is 
the set of all rational maps $Q$ that are Hurwitz equivalent to $R$. The  
topology we are considering on $H(R)$ is the compact-open topology.

Let $f:S_2\rightarrow S_2'$  be a representative of a point in $T(S_2)$ with 
Beltrami coefficient $\mu$ and fixing the points $b_i$. Let $R_*(\mu)$ be the 
pull-back of $\mu$ under $R$, and  $h_{f}$ be the solution, defined on the Riemann sphere, of the 
Beltrami equation for the coefficient $R_*(\mu)$, and take $S'_1=h_f(S_1)$. 
Let us define the map $\tau:T(S_2)\rightarrow H(R)$ so that $\tau(f)$ is
the rational map making the following diagram commutative

\begin{displaymath}
    \xymatrix{
        S_ 1 \ar[d]_{R} \ar[r]^{h_{f}}  & S'_1
\ar[d]^{\tau(f)} \\
             S_2  \ar[r]^{f}   &  S'_2 }
\end{displaymath} 

The map $\tau$ is well defined and continuous, since the solution of the 
Beltrami equation depends analytically on $\mu$.  We call the space 
$H_\tau(R)=\tau(T(S_2))$, the 
\textit{reduced Hurwitz space} of $R$. The closure of the bi-orbit by the M\"obius group 
of $H_\tau(R)$ is the whole $H(R).$ The group $PSL(2,\C)$ also acts by conjugation on 
$H(R).$ The space of orbits by conjugation fibres over the reduced Hurwitz 
space $H_\tau(R)$.

Let $f$ be an element of the Mapping Class Group $MCG(S_2)$
such that $f(b_i)=b_i$. We say that $f$ is \textit{liftable} with respect to 
$R$ if there exist a map $g:S_1\rightarrow S_1$, such that $g(a_i)=a_i$ and 
makes the following diagram commutative:

\begin{displaymath}
    \xymatrix{S_1
         \ar[d]_{R} \ar[r]^{g} & S_1
\ar[d]^R \\
          S_2 \ar[r]^f  &  S_2  }
\end{displaymath}

In this case, we say that $g$ is the lifted map of $f$ with respect to $f$. 
Let $\mathcal{G}$  be the subgroup of the mapping class group $MCG(S_2)$ which
consists of all liftable elements with respect to $R$. We identify 
$H_\tau(R)$ with the space $T_2(S_2)/\mathcal{G}$ as the following theorem suggests.

\begin{theorem} The space $H_\tau(R)$ is continuously bijective to
$T(S_2)/\mathcal{G}.$
\end{theorem}

\begin{proof}

We will show that $\tau(\phi_1)=\tau(\phi_2)$ if and only if $\phi_2^{-1}\circ
\phi_1\in \mathcal{G}$. 

Let $f=\phi_2^{-1}\circ \phi_1$, and assume that $f$ belongs to $\mathcal{G}$ and let $g$ be the lifted map 
of $f$ with respect to $R$, so if $h_{\phi_1}$ and $h_{\phi_2}$ are the maps associated with $\tau(\phi_1)$
and $\tau(\phi_2)$ respectively, we have $h_{\phi_2}\circ g=h_{\phi_1}$ by the 
normalization of $g$.
Hence $$\phi_2\circ f \circ R=\phi_1\circ R=\tau(\phi_1)\circ h_{\phi_1},$$ on the 
other hand, $$\phi_2\circ f \circ R=\phi_2\circ R\circ g=\tau(\phi_2)\circ 
h_{\phi_2}\circ g$$ so we get
$$\tau(\phi_1)=\tau(\phi_2).$$ 

Reciprocally, if $\tau(\phi_1)=\tau(\phi_2)$,  then $f=\phi_2^{-1}\circ 
\phi_1$ fixes the points $b_i$. Since $h_{\phi_2}^{-1}\circ h_{\phi_1}$ is a lift of 
$f$ with respect to $R$, the map $f$ belongs to $\mathcal{G}.$

\end{proof}

The following theorem gives a description of compact subsets of $H(R)$. 
For a quasiconformal map $f$, let $K_f(z)$ be the distortion of $f$ at the 
point $z.$

\begin{theorem}
 If $\{f_i\}$ is a family of quasiconformal maps on the sphere $\bar{\C}$ 
fixing the points $b_1,b_2$ and $b_3$. Let
 $$A_n=\{z:  K_{f_i}(z)\geq n \textnormal{ for } i \textnormal { big
enough.}\}$$
Assume that $A_\infty=\overline{\bigcap A_n}$ is a compact subset of $S_2$. So 
we have:
\begin{itemize}

\item If $A_\infty$ does not separate the critical values of $R$, then the
family $\{[f_i]\}$ is bounded in $T(S_2)$.  

\item If there exist a domain $D_0$  contained in $S_2\setminus
A_\infty$ such that $D_0$ contains at least two of the points 
$b_i$ and $W_0=R^{-1}(D_0)$ is connected then the
respective classes $\{\tau([f_i])\}$ 
are bounded in $Rat_d(\C).$
\end{itemize}

\end{theorem}

\begin{proof}
For the first item, let $U$ be a neighborhood of $A_\infty$ such that is 
compactly contained in $S_2$, does not separate $S_2$, and such that every 
component of $U$ is simply connected with analytic boundary.  

The restrictions of $f_i$ on $\partial U$ are quasisymmetric maps with uniform 
bound of distortion. Using Douady-Earle extension, the maps $f_i|_{\partial U}$ extend
to maps  $\tilde{f}_i$, defined on the interior of $U$,  with 
uniformly bounded distortion. Since $f_i$ and $\tilde{f}_i$ have the same 
values on the boundary, the maps $f_i$ and $\tilde{f}_i$ are homotopic, and 
define the same points in $T(S_2)$. Hence the family $f_i$ have uniformly 
bounded Beltrami coefficients, so defines a bounded set in $T(S_2).$

Since $Rat_d(\C)$ can be identified with an open and dense subset of the 
projective space,  it is enough to prove that all the limit maps of 
$\{\tau(f_i)\}$ are rational maps of 
the same degree. Let $g_i$ be quasiconformal automorphisms of $\bar{\C}$ fixing 
the points $a_i$ and such that $f_i\circ R=\tau(f_i)\circ g_i.$ 
Under the conditions of the second item, the accumulation functions of the families 
$\{f_i|_D\}$ and 
$\{g_i|_{W_0}\}$ are  non constant quasiconformal
functions. Let $R_\infty$ be a rational map which is accumulation point of the maps $\tau(f_i)$,
then there exist two non constant quasiconformal functions $f_\infty$ and
$g_\infty$ and domains $O = h_\infty(D)$ and $X = g^{-1}_\infty(W_0)$ on which 
$f_\infty\circ 
R=R_\infty\circ g_\infty$. Then
$deg(R_{\infty|O}) = deg(R_{|W_0})=deg(R)$,  as we wanted to prove.
       
\end{proof}

%
%
%

\subsection{Bers slices}

Let ${\Delta}=\{z:|z|<1\}$ and ${\Delta}^*=\{z:|z|>1\}.$ We denote the boundary 
of $\Delta$ by $\mathbb{S}^{1}.$  Now let us consider again the rational map $R$ and surfaces $S_i$. Let  
$\Gamma_i$  be Fuchsian groups that uniformizes the surfaces $S_i$. By the 
Monodromy Theorem there exist a M\"obius map
$\alpha$ such that $\tilde{\Gamma}=\alpha\Gamma_1\alpha^{-1}$ is a 
subgroup of $\Gamma_2$. If $p:\Delta \rightarrow S_1\cong \Delta/\tilde{\Gamma}$
is the orbit projection. Then defining $\pi_1=p\circ \alpha^{-1}$ gives the
following diagram

\begin{displaymath}
    \xymatrix{
        \Delta \ar[d]_{\pi_1} \ar[r]^{Id} &
\ar[d]^{\pi_2} \Delta \\
             S_1\cong \Delta/\Gamma_1  \ar[r]^{R}  &  
S_2\cong \Delta/\Gamma_2.}
\end{displaymath}

Moreover, $[\Gamma_2:\Gamma_1]=deg(R)$ since $R$ is a covering.  We call the 
pair $(\Gamma_1,\Gamma_2)$ a uniformization of $R.$ Simultaneously we have
uniformization of the surfaces $S^*_i=\Delta^*/\Gamma_i$ and the map
$Q:S_1^*\rightarrow S_2^*$ given by $Q(z):=\overline{R(\bar{z})}$. Now
these groups are acting on the complement of the unit disk $\Delta^*$.

Let  $D(\Gamma_2)$ be the space of groups $\Gamma$ such that there 
exist a
quasiconformal map $f$ such that $\Gamma=f\circ \Gamma_2 \circ f^{-1}$ 
and such that the Beltrami differential $\mu_f=0$ in $\Delta^*$. Now put
$$Def(\Gamma_2)=D(\Gamma_2)/PSL(2,\C).$$
  Analogously, we
define $Def^*(\Gamma_2)$ as the space of deformations of $\Gamma$ on
$\Delta^*.$ By Bers Theorem, both spaces $Def(\Gamma_2)$ and 
$Def^*(\Gamma_2)$ have compact closure on the space of classes of 
faithful and discrete representations 
$$\Gamma_2\hookrightarrow PSL(2,\C).$$ These closures of $Def(\Gamma_2)$ and
$Def^*(\Gamma_2)$  are called Bers slices of the Teichm\"uller space. In these
cases, the Bers slices consist of function groups with a simply connected 
invariant component. Geometrically finite groups contained in the Bers slice 
are either quasifuchsian or cusps. Definitions and properties can be found in 
the papers by  Bers \cite{bers1970boundaries}, by Maskit 
\cite{Maskitbounda} and by McMullen \cite{McMullenCusps}. 
If a group $G$ in the boundary of the Bers slice has a connected
region of discontinuity, then $G$ is totally degenerated. By theorems of Bers,
Maskit and McMullen (see \cite{bers1970boundaries}, \cite{Maskitbounda} and 
\cite{McMullenCusps}), totally degenerated groups and cusps are both dense on 
the boundary of the Bers slice. 

Any group $G$ in the Bers slice defines a 3-hyperbolic manifold $M(G)$
with boundary. Given a uniformization $(\Gamma_1, \Gamma_2)$ of $R$, let us 
consider  $M_1:=M(\Gamma_1)$ 
and $M_2:=M(\Gamma_2)$ the associated 3-hyperbolic manifolds with boundary. 
Then the inclusion $\Gamma_1$ in $\Gamma_2$ 
defines a M\"obius map $$F:M_1\rightarrow M_2.$$ The 
restriction of $F$ to the boundary components of $M_1$ define maps 
which are rational in coordinates. Let $\Sigma_i \subset \partial M_i$ be the 
respective invariant components of $\Gamma_i$. Then the map 
$F:\Sigma_1\rightarrow \Sigma_2$ belongs to either the Hurwitz space $H(R)$ or 
$H(Q)$.  

If a group $G$ in the Bers slice $Def(\Gamma_2)$ is geometrically finite, then 
$G$ is a cusp or quasifuchsian. In those situations there is a boundary 
component $S$ in $M(G)$ conformally equivalent to $\Delta^*/\Gamma_2$.  In the 
case that there are cusps on the group, we regard the set of all 
components as a connected surface with nodes $\tilde{S}$.

\textbf{Question.} Assume that $G_i$ converges, in Bers slice, to a totally
degenerated group $G$. Is it true that the accumulation set of the associated
rational maps may contain constants maps?

Now we are ready to proof the main result of this section. Let us begin with 
the following definition. Let $G$ be a totally degenerated group. Then we call  
the group $G$ \textit{acceptable} for the rational map $R$ if and only if  the 
following conditions hold:

\begin{itemize}
\item There are two uniformizable M\"obius orbifolds $S_i$ supported on the 
Riemann sphere, such that 
$R:S_1\rightarrow S_2$ is a holomorphic covering.
 \item If $\Gamma_2$ is Fuchsian group uniformizing $S_2$, then $G$ belongs 
to $Def^*(\Gamma_2)$.

\item The manifold $M(G)$ is homeomorphic to $\partial M(G)\times 
\mathbb{R}_+$, here $\mathbb{R}_+$ denotes the set of non-negative real numbers.

\end{itemize}
Let $\pi:\mathbb{H}^3\cup \Omega(G)\rightarrow M(G)$ be the orbit 
projection.  Under the homeomorphism of the last item, let us define, for every 
$t$ in $\mathbb{R}_+$ the set $\Omega(G)_t=\pi^{-1}(\partial M\times {t})$. 
Hence, the space $\mathbb{H}^3\cup \Omega(G)$ is foliated by the sets 
$\Omega(G)_t$ and there exists a continuous family of homeomorphisms 
$f_t:\Omega(G)\rightarrow \Omega(G)_t$  which commutes with $G$ and $f_0=Id.$

Let $\mathbb{B}^3$ denote the unit ball model for the hyperbolic space. Given a 
rational map $R$, we define the radial extension $\tilde{R}$ as follows. For
every $\lambda \in [0,1]$ and $(x,y,z)\in \mathbb{R}^3$, let
$H_\lambda(x,y,z)=(\lambda x, \lambda y, \lambda z)$. Then we have
$$\bar{\mathbb{B}}^3=\bigcup_{\lambda \in [0,1]}H_\lambda (\partial
\mathbb{B}^3).$$ Now
define $\hat{R}(0,0,0)=(0,0,0)$ and for $v\in\bar{\mathbb{B}}^3$, different from
$0$, define $\hat{R}(v)=H_{\|v\|}\circ R \circ H^{-1}_{\|v\|}$ where $\|v\|$
denotes the euclidean distance to $v$ from the origin.

\begin{theorem}
If there exist an acceptable group for $R$ then the radial extension of $R$ is
geometric.
\end{theorem}

\begin{proof}
 Let $G_2$ be an acceptable group for $R$, then there exist 
$\phi:\Omega(G_2)\rightarrow \Delta$ such that induces an isomorphism 
$\phi_*:G_2\rightarrow \Gamma_2$, then $G_2$ has a finite index subgroup 
$G_1=\phi_*^{-1}(\Gamma_1)$, such that map 
$\alpha(R):M(G_1)\rightarrow M(G_2)$ is M\"obius. Moreover, the manifold 
$M(G_1)$ is a manifold homeomorphic to $\partial M(G_1)\times \mathbb{R}_+$. We 
have to 
show that the radial extension is equivalent to $\alpha(R)$, so it is 
geometric. Again, each $M(G_i)$ is homeomorphic to $S_i\times 
\mathbb{R}_+$, and the horizontal foliation in $M_1$ is the pull back by 
$F:M_1\rightarrow M_2$ of the horizontal foliation in $M_2$. Hence there exist a 
covering 
$\phi:S_1\times \mathbb{R}_+\rightarrow S_2\times \mathbb{R}_+$ and
homeomorphisms $h_1, h_2$  such that the following
diagram commutes

\begin{displaymath}
    \xymatrix{
        M_1 \ar[d]_{h_1} \ar[r]^{F} &
\ar[d]^{h_2} M_2 \\
             S_1\times \mathbb{R}_+  \ar[r]^{\phi}  &  
S_2\times \mathbb{R}_+}
\end{displaymath}
such that $h_1(x)=x$ and $h_2(y)=y$ for all $x$ in $S_1$ and $y$ in $S_2$.
Hence $\phi(x)= F(x)=R(x)$ for $x$ in $S_1$. There
are two families of homeomorphisms $\psi_t$ and $\chi_t$ such that 

\begin{displaymath}
    \xymatrix{
        S_1\times\{t\} \ar[r]^{\phi} &
 S_2\times \{t\}\\
             S_1\times \{0\} \ar[u]_{\psi_t}  \ar[r]^{\phi}  & \ar[u]^{\chi_t}
S_2\times\{0\}}
\end{displaymath}
So that $\phi$ preserves the parameter $t$. Now consider a homeomorphism 
$k:[0,\infty]\rightarrow [0,1]$ such that $k(0)=1$ and $k(\infty)=0$. For 
$i=0,1$, let us identify the sets $S_i\times 0$ with the corresponding $S_i$ 
on the Riemann sphere. Hence  we define two homeomorphisms
$\psi$ and $\chi$ such that $\psi(x,t)=H_{k(t)}(\psi^{-1}_t(x))$ and
$\chi(x,t)=H_{k(t)}(\chi^{-1}_t(x))$, where $H_t(x,y,z)=(tx,ty,tz)$. Since 
$\psi$ and 
$\chi$ are the identity on the boundary, these homeomorphisms uniformize the 
extension of $\phi$ over $S_1$ and $S_2.$ 

\end{proof}

We have not found a reference that shows that the manifold of a totally 
degenerated group in the Bers slice of a finitely 
generated group is always a product. However Michael Kapovich kindly gave us
arguments to show this happens. The arguments are based upon a work of 
Waldhausen and the solution of the Tame Conjecture.

\section{Schottky type extensions of rational maps}

In this section, we prove that any map in the Hurwitz space of a Blaschke map
has an extension of Schottky type that satisfies the properties (1), (2) and 
(3) in the introduction. We also construct an extension of Blaschke 
maps that satisfy almost all five conditions: this extension does not 
satisfies condition (4). However, if we take the group of M\"obius 
transformations preserving the unit circle instead of $PSL(2,\C)$, then the 
modified condition (4) holds.

A \textit{Blaschke map} $B:\hat{\C} \rightarrow \hat{\C}$ is a rational map that
leaves the unit disk $\Delta$ invariant.  If $d$ is the degree of $B$, then 
there exist $\theta \in [0,2\pi]$ and $d$ points $\{a_1,...,a_d\}$ in $\Delta$ 
such that $$B(z)=e^{i\theta} 
\left(\frac{z-a_1}{1-\bar{a}_1z}\right)...\left(\frac{z-a_d}{
1-\bar{a}_d z}\right).$$ 
Let us denote by $B_{1}=B|_\Delta$ and $B_{2}=B|_{\Delta^*}$, then  
$B_{2}(z)=\frac{1}{\bar{z}} \circ B_{1}(z) \circ \frac{1}{\bar{z}}$.

Let $CV(B)$ be the set of critical values of $B$ then define 
$S_2=\bar{\C}\setminus CV(B)$ and $S_1=B^{-1}(S_2)$. Thus $B:S_1\rightarrow 
S_2$ is a holomorphic covering and the surfaces $S_i$ are symmetric with 
respect to $\mathbb{S}^1.$

The class of Blaschke maps allow us to build a specific topological construction
based upon Schottky coverings of Riemann surfaces. What is special about 
Blaschke maps is that every Blaschke map commutes with the involution 
$\tau(z)=\frac{1}{\bar{z}}$. Recall that Fuchsian groups of second type are 
Schottky type Fuchsian groups.

\begin{theorem}\label{thm.Blasc.Scho} Given a Blaschke map 
$B$, such that $B:S_1\rightarrow S_2$ is a covering and the surfaces $S_i$ are 
symmetric surfaces with respect to $\mathbb{S}^1$. There are two
Fuchsian groups of second type $\Gamma_1$ and $\Gamma_2$ such that 
$\Omega(\Gamma_i)/\Gamma_i=S_i$, where $\Omega(\Gamma_i)$ is the 
discontinuity set of $\Gamma_i$ for $i=1,2$. Furthermore, there exist an 
M\"obius map $\alpha:\Omega(\Gamma_1)\rightarrow \Omega(\Gamma_2)$ making
following diagram commutative

\begin{displaymath}
    \xymatrix{\Omega(\Gamma_1)
         \ar[d]_{\pi_1} \ar[r]^{\alpha} & 
\ar[d]^{\pi_2}  \Omega(\Gamma_2) \\ S_1
           \ar[r]^B  & S_2.  }
\end{displaymath}

Also $[\Gamma_2:\alpha \Gamma_1 \alpha^{-1}]=deg (B).$

\end{theorem}

\begin{proof}
For $i=1,2$ let $\Delta_i=\bar{\Delta}\cap S_i$ 
and $\Delta^*_i=\bar{\Delta}^*\cap S_i$.

The Simultaneous Uniformization Theorem
of Bers \cite{BersSimUniform}, ensures that there exist a Fuchsian group 
$\Gamma_2$ acting in $\hat{\C}$, with ${\Delta}/\Gamma_2={\Delta}_{2}$ and 
${\Delta}^*/\Gamma_2=\Delta^*_2$. The limit set 
${\Lambda}(\Gamma_2)  $ is contained  in $ \mathbb{S}^{1}$. 

Similarly, there is a Fuchsian group $\Gamma_1$ such that
${\Delta}^*/\Gamma_1=\Delta^*_{1}$ and ${\Delta}/\Gamma_1={\Delta}_{1}$ 
and ${\Lambda(\Gamma_1)} \subset \mathbb{S}^{1}$.

The map $B_{i}$ lifts to M\"{o}bius maps ${\alpha}_{1}:{\Delta} \rightarrow
{\Delta}$ and ${\alpha}_{2}:{\Delta}^* \rightarrow
{\Delta}^*$. Moreover, since ${\Omega}_{2}$ is a Riemann surface
anti-conformally equivalent to ${\Omega}_{1}$, we can choose ${\alpha}_{2}$ 
such that ${\alpha}_{2}(z)= \frac{1}{\bar{z}} \circ {\alpha}_{1} \circ
\frac{1}{\bar{z}}$ and these maps agree at 
$\mathbb{S}^{1}\setminus{\Lambda}(G)$. 
Being ${\Lambda}(G)$ a Cantor set, then the map ${\alpha}:
{\hat{\C}}\setminus{\Lambda}(G) \rightarrow {\hat{\C}}\setminus{\Lambda}(G)$ 
defined as ${\alpha}|{\Delta}_{i}={\alpha}_{i}$ extends to a  M\"{o}bius 
map $\alpha$ defined on the Riemann sphere. \end{proof}

By Theorem \ref{thm.Blasc.Scho}, a Blaschke map admits a Poincar\'e extension 
which follows from the diagram below:

\begin{displaymath}
    \xymatrix{
        \mathbb{B}^3 \ar[d] \ar[r]^{\alpha} &
\ar[d] \mathbb{B}^3 \\
             \B^3/\Gamma_1  \ar[r]^{\hat{B}}  &   \B^3/\Gamma_2.}
\end{displaymath}

Let us observe that $\Gamma_1$ and $\Gamma_2$ are  Schottky type groups with 
parabolic generators. Hence ${\B}^{3}/\Gamma_1$ and ${\B}^{3}/\Gamma_2$ 
are homeomorphic to the complement in ${\B}^{3}$ of a finite number of  
geodesics connecting symmetric perforations of the surfaces $S_i$.

The arguments in Theorem \ref{thm.Blasc.Scho} work in a more general situation. 
The key facts are Bers Simultaneous Uniformization Theorem and the symmetry of 
respective orbifolds. So we have

\begin{corollary}
 Let $W_1$ and $W_2$ be any two given connected symmetric orbifolds 
supported on the Riemann sphere. If  $B:W_1\rightarrow W_2$ is a 
covering symmetric with respect to $\mathbb{S}^1$, then the conclusion of  
Theorem \ref{thm.Blasc.Scho} still holds for the covering $B$.
\end{corollary}

Let $G\mathcal{O}(P(B))$ be the grand orbit of the postcritical set $P(B)$ and 
take $S=\bar{\C}\setminus \overline{G\mathcal{O}(P(B))}$, then $S$ is an open 
disconnected surface consisting of two components $D=S\cap \Delta$ and 
$D^*=S\cap \Delta^*$ and $B:S\rightarrow S$ is a holomorphic self-covering.

\begin{corollary}\label{cor.Bla.dyn}
 There exist a Fuchsian group $\Gamma$ and $\alpha$ in $PSL(2,\mathbb{R})$ such that
$\alpha_*(\Gamma)=\alpha\Gamma\alpha^{-1}$ is a subgroup of $\Gamma$. 
Moreover, we have that  $\Omega(\Gamma)=\Delta\cup \Delta^*$ and 
$\Delta/\Gamma=D$. Finally, for every $n$ the following diagram is commutative
$\Delta^*/\Gamma=D^*$  
\begin{displaymath}
    \xymatrix{
       \Omega(\Gamma) \ar[r]^{\alpha^n} \ar[d]_{\pi_1} &
\ar[d]^{\pi_1} \Omega(\Gamma) \\
             S  \ar[r]^{B^n}  & S.}
\end{displaymath}
 
\end{corollary}
\begin{proof}
The proof is essentially the same as in Theorem \ref{thm.Blasc.Scho} using the
symmetry of the surfaces plus the fact that $B$ defines a self-covering of 
the surface $S$.
\end{proof}

As noted after Theorem \ref{thm.Blasc.Scho}, 
we have that the Poincar\'e extension of $B$ is an endomorphism of 
$\mathbb{B}^3/\Gamma$, so this Poincar\'e extension is dynamical. 

Again, the argument in the corollary above can be generalized as in the 
following corollary:

\begin{corollary}
Given a Blaschke map $B$, let $A$ be a completely invariant symmetric closed subset of 
$\bar{\mathbb{C}}$ which contains all critical points of $B$. Assume that 
$\bar{\mathbb{C}}\setminus A= W$ consists of exactly two components 
$U=W\cap \Delta^*$ and $U=W\cap \Delta^*$. 
Then $\{B^n:U^*\rightarrow U^*\}$ and $\{B^n:U\rightarrow U\}$ are 
semigroups of  coverings and the Poincar\'e extension in this situation is 
dynamical.
\end{corollary}

Now let us consider the case of decomposable Blaschke maps.  Assume that
$B=B_1\circ B_2$ is a decomposable Blaschke map where $B_1$ and
$B_2$ are Blaschke maps. Then if $Q=B_2\circ B_1$ there are semiconjugacies. 

$$\begin{CD} \C @>B=B_1\circ B_2>> \C\\
@V{B_2}VV @VV{B_2}V\\
\C @>Q=B_2\circ B_1>> \C\\
@V{B_1}VV @VV{B_1}V\\
\C @>{B}>> \C.\end{CD} $$

\begin{corollary}\label{cor.10}

If $B=B_1\circ B_2$ and $Q=B_2\circ B_1$ there exist two Fuchsian groups 
$\Gamma(B)$ and
$\Gamma(Q)$ satisfying the conditions of Corollary \ref{cor.Bla.dyn} and
$\alpha_B$ and $\alpha_Q$ in $PSL(2,\mathbb{R})$. So that there are two
elements $\beta_1$ and $\beta_2$ in $PSL(2,\mathbb{R})$ with
$\alpha(B)=\beta_1\circ \beta_2$ and $\alpha(Q)=\beta_2\circ \beta_1.$ Such
that the following diagrams are commutative.

\begin{displaymath}
\xymatrix{
& \Omega(\Gamma_B) \ar[rr]^{\beta_2}\ar'[d]^{P_B}[dd]
& &\Omega(\Gamma_Q) \ar[dd]^{P_Q}
\\
\Omega(\Gamma_B)\ar[ur]^{\alpha_B}\ar[rr]\ar[dd]^{P_B}
& &\Omega(\Gamma_Q)\ar[ur]_{\alpha_Q}\ar[dd]
\\
& S_B \ar'[r][rr]_{B_2}
& & S_Q
\\
S_B\ar[rr]_{B_2}\ar[ur]_{B}
& & S_Q \ar[ur]_{Q}
}
\end{displaymath}

\begin{displaymath}
\xymatrix{
& \Omega(\Gamma_Q) \ar[rr]^{\beta_1}\ar'[d]^{P_Q}[dd]
& &\Omega(\Gamma_B) \ar[dd]^{P_B}
\\
\Omega(\Gamma_Q)\ar[ur]^{\alpha_Q}\ar[rr]\ar[dd]^{P_Q}
& &\Omega(\Gamma_B)\ar[ur]_{\alpha_B}\ar[dd]
\\
& S_Q \ar'[r][rr]_{B_1}
& & S_B
\\
S_Q\ar[rr]_{B_1}\ar[ur]_{Q}
& & S_B \ar[ur]_{B}
}
\end{displaymath}

\end{corollary}\label{cor.Blasc.decomp}
\begin{proof}
 Note that $B_2:S_B\rightarrow S_Q$ and $B_1:S_Q\rightarrow S_B$ define
 coverings, hence there are M\"obius maps 
$\beta_1:\Omega(\Gamma_Q)\rightarrow \Omega(\Gamma_B)$ and
$\beta_2:\Omega(\Gamma_B)\rightarrow \Omega(\Gamma_Q)$ which make the diagrams
commutative.
\end{proof}

As a consequence of Corollary \ref{cor.10} we have the following conclusion.

\begin{proposition}\label{prop.11}
Let $B=B_1\circ B_2$, let $\hat{B}$ be a dynamical extension, then there are
Poincar\'e extensions  
$\hat{B}_1$ and $\hat{B}_2$ of $B_1$ and $B_2$ respectively, such that 
$$\hat{B}=\hat{B}_1\circ\hat{B}_2.$$
Moreover, there exist $\hat{Q}$ a dynamical extension of $Q:=B_2\circ B_1$ such
that 
$$\hat{Q}=\hat{B}_2\circ\hat{B}_1.$$
\end{proposition}

The following theorem summarize the results above and show that the 
corresponding extensions are geometric.

\begin{theorem}\label{th.Blasch.dyn}

Let $B$ be a Blaschke map, then:

\begin{itemize}
 \item[i)] The extension constructed in Theorem \ref{thm.Blasc.Scho} is 
geometric.

\item[ii)]  The dynamical extension constructed in Corollary \ref{cor.Bla.dyn} 
is geometric. 

\item[iii)] The extensions in Corollary \ref{cor.Blasc.decomp} are all 
geometric.

\end{itemize}
Each of the extensions in items (i)-(iii) is conformally natural with respect 
to the group of M\"obius transformations that leaves the unit circle invariant.

\end{theorem}
\begin{proof}
We will show item (i), the proof the other items apply similar arguments to
the extensions constructed in Corollary \ref{cor.Bla.dyn} and Proposition 
\ref{prop.11}. Again, the important feature is that the corresponding surfaces 
are symmetric with respect to the unit circle.
According to Theorem \ref{thm.Blasc.Scho} there are manifolds
$M_1$ and $M_2$, M\"obius projections $p_1$ and $p_2$, and a Poincar\'e 
extension $\hat{B}$ of 
$B$ such
that the following diagram is commutative
\begin{displaymath}
    \xymatrix{\mathbb{H}^3\cup \Omega(\Gamma_1)
         \ar[d]_{p_1} \ar[r]^{ Id} & 
\ar[d]^{p_2}  \mathbb{H}^3\cup\Omega(\Gamma_2) \\ M_1
           \ar[r]^{\hat{B}}  & M_2.  }
\end{displaymath}
such that  $p_i|_{\Omega(\Gamma_i)}=\pi_i$.
Now, we construct   universal coverings $q_i$ which maps $\mathbb{H}^3\cup
\Omega(\Gamma_i)$ in $\bar{\mathbb{B}}^3$ and $q_i(x)=q_i(y)$ if, and only if, 
there exist a $\gamma_i$ in $\Gamma_i$ with $\gamma_i(x)=y$, so that 
$$q_i|_{\Omega(\Gamma_i)}=p_i|_{\Omega(\Gamma_i)}= \pi_i.$$ By Theorem
\ref{thm.Blasc.Scho}, the group  $\Gamma_2$ acts on the unit disk which
belongs to the boundary of $\bar{\mathbb{B}}^3 \cap \mathbb{R}^2 $ so that
$\pi_2(\Delta)=S_2\cap \Delta$ also belongs to the boundary of 
$\bar{\mathbb{B}}^3\cap \mathbb{R}^2 $. Let $\tau_\phi$ be the 
M\"obius rotation, in 
$\mathbb{R}^3$, with respect to $\partial \Delta$ of angle $\phi$.
Then $$\bar{\mathbb{B}}^3=\cup_{0\leq \phi \leq \pi} \tau_\phi(\Delta)$$  and
$\tau_\pi:\Delta\rightarrow \Delta^*$ is the map $z\mapsto 1/\bar{z}$ 
in the holomorphic coordinate of $\Delta.$ Define $q_i(z,\phi)=(\tau_\phi\circ
\pi_i(z))$ such that $\tau_\phi$ commutes with 
$Aut(\Delta)\simeq PSL(2,\mathbb{R}).$
Furthermore, $\tau_\phi$ commutes with any M\"obius map that leaves the unit
circle invariant (for instance $z\mapsto 1/z$).

Then $M_i=q_i(\mathbb{H}^3\cup \Omega(\Gamma))$ are subsets of 
$\bar{\mathbb{B}}^3$  and the respective Poincar\'e
extension is conformally natural with respect to the group of M\"obius map that 
leaves the unit circle invariant.

\end{proof}
 \begin{corollary}
 If in  Proposition \ref{prop.11}, $B_1=B_2$ then $\hat{B}_1=\hat{B}_2$. Using
induction 
If $B=B_1^n$ then for every dynamical extension $\hat{B}$ of $B$, there exist a 
dynamical extension $\hat{B}_1$ of $B_1$ such that $\hat{B}=\hat{B}_1^n$.
\end{corollary}

In the case (i) of Theorem \ref{th.Blasch.dyn}, let
${Q}_i:\mathbb{H}^3 \rightarrow \bar{\mathbb{B}}^3$ be other extensions of the
projections $\pi_i$, then there are continuous maps $h_i:M_i\rightarrow 
\bar{\mathbb{B}}^3$ such that $Q_i=h_i\circ q_i.$ Where $q_i$ are the
extensions constructed on the proof of the Theorem \ref{th.Blasch.dyn}.

Let us put $K={Q}_2\circ{Q}^{-1}_1$ where the composition is defined. If 
$\hat{B}$ is the
geometric extension of $B$ from Theorem \ref{th.Blasch.dyn}, then $K\circ 
h_2=h_1\circ \hat{B}$. 

In the case (ii) of Theorem \ref{th.Blasch.dyn}, assume $Q_i$ is another 
extension of $\pi_i$. Again, put $K={Q}\circ \alpha_B\circ{Q}^{-1}$.
If $K$ is a map, then $K$ is semiconjugated to $\hat{B}.$

\begin{theorem}\label{semigroup.Blasch}
 Let $S$ be a semigroup of Blaschke maps, then the extension constructed in
Theorem \ref{th.Blasch.dyn} defines in $S$ a geometric homomorphic conformally
natural extension preserving degree if, and only if, $S$ does not intersect the
bi-orbit of $f(z)=z^n$ with respect to $Aut(\Delta)$.
\end{theorem}
\begin{proof}

If the extension is not conformally natural with respect to $PSL(2,\C)$ then 
there are two elements $B_1$ and $B_2$ in $S$ with two maps $g_1$ and
$g_2$ in $PSL(2,\C)$ such that $B_1\circ g_1=g_2\circ B_2$. Then by Theorem
\ref{th.Blasch.dyn}, the maps $g_1$ and $g_2$ cannot leave the unit circle 
invariant. 
Then there two circles $C_1=g_1(\mathbb{S}^1)$ and $C_2=g_2(\mathbb{S}^1)$
with $B_1^{-1}(C_2)=C_1.$ Let us show that the circles $C_i$ do not intersect
$\mathbb{S}^1.$ Assume that there is $x$ in $ C_1\cap \mathbb{S}^1$ then 
$C_2$ intersects $\mathbb{S}^1$ in all the preimages of $x$ with respect to
$B_1.$  But this is possible only if the preimage of $x$ under $B_1$ is a
single critical point, but a Blaschke map cannot have critical points on
$\mathbb{S}^1$. Therefore, $C_1$ and $C_2$ cannot intersect $\mathbb{S}^1.$
Then $C_1$ and $\mathbb{S}^1$ bound an annulus. By the reflection principle, 
we have that $B_1$ has in the unit disk a critical point of multiplicity $d-1.$
Hence either $B_1$ or $1/B_1$ belongs to the bi-orbit of $z^n$ with respect of 
$Aut(\Delta)$, but  the extension in Theorem \ref{thm.Blasc.Scho} is
compatible with the map $z\mapsto 1/z$. Hence the conclusion of the theorem
holds. Reciprocally, the extension of Theorem \ref{thm.Blasc.Scho} is not
conformally natural on the map $z^n$.

\end{proof}

\begin{corollary}
 Assume $B$ is a Blaschke map of the form $B=g_1 \circ z^d \circ g_2$
such that 
$g_2(z)\neq e^{i\alpha}\circ g_1^{-1}(z)$, then the extensions constructed on 
Theorem \ref{th.Blasch.dyn} of 
$\langle B^n\rangle$ are conformally natural whenever $n\geq 2$.
\end{corollary}
\begin{proof}
 If for $n\geq 2$, $B^n$ does not satisfies the conditions of Theorem
\ref{semigroup.Blasch}, then $B^n$ should be in the bi-orbit of $z^{d^n}$.
Hence $B^n$ has one critical point and one critical value, this implies that
the critical point $x$ of $B$ is fixed. Hence $g_1\circ z^d \circ g_2(x)=x$
then $g_1(0)=g_2^{-1}(0)$, which implies  $g_2\circ g_1(0)=0$ and
$g_2=e^{i\alpha}\circ g_1^{-1}(z)$.
\end{proof}

\subsection{Geometric extensions in Hurwitz spaces}

Let us recall that a branched covering $f:\hat{\C} \rightarrow
\hat{\C}$ of degree $d$ is in general position,  if the number of critical 
points is the same than the number of critical values and equal to
$2d-2$.
According to \cite{SK} a theorem of Luroth and Clebsch states that:

\begin{lemma}\label{Lem.Luroth.Clebsch}Any two branched coverings 
$f_{i}:\hat{\C} \rightarrow
\hat{\C}$ of the same degree in general position are Hurwitz equivalent.
\end{lemma}

The existence of a Schottky type extension is a property of the whole Hurwitz 
space as we show in the following.

\begin{lemma}\label{lem.branch.Hurw} Let $B$ be a map with a Schottky type 
geometric extension and $R$ a 
rational map in $H(B)$, then $R$ has also an extension of Schottky
type.

\end{lemma}

\begin{proof} Let $\Gamma_1$ and $\Gamma_2$ be the uniformizing groups for $B$. 
Let $\alpha : \C \rightarrow \C$ be the M\"obius map extending $B$. By 
definition of $H(B)$, there are two quasiconformal maps $f,g$ on the Riemann 
sphere such that $f\circ B=R\circ g.$ Solving the Beltrami equation,  we get 
quasiconformal extensions $\tilde{f}:\C\rightarrow \C$ and 
$\tilde{g}:\C\rightarrow \C$ of $f$ and $g$ respectively, such that 
$\beta=\tilde{g}^{-1}\circ \alpha \circ \tilde{f}$ is a M\"obius  map extending 
 $R$.  Let us assume first that $\tilde{f}$ and $\tilde{g}$ have small 
distortion, then by a theorem in \cite[Th. 5]{DouadEarl} each map, $\tilde{f}$ 
and $\tilde{g}$, admits a homeomorphic extension, say $\hat{f}$ and $\hat{g}$, 
to $\mathbb{B}^3$ compatible with the 
groups $\Gamma_1$ and $\Gamma_2$. Hence we obtain two 
Kleinian groups $\tilde{\Gamma}_1=\hat{f}\circ \Gamma_1 \circ \hat{f}^{-1}$ 
and $\tilde{\Gamma}_2=\hat{g}\circ \Gamma_2\circ \hat{g}^{-1}$ with manifolds 
$M(\tilde{\Gamma}_1)$ and $M(\tilde{\Gamma}_2)$ that extend the map $R$.
To see that this extension is geometric we have to embed each manifold  $M(\tilde{\Gamma}_1)$ 
and $M(\tilde{\Gamma}_2)$ into $\mathbb{B}^3$, such that the image of these 
embeddings are the complement of a finite set of quasigeodesics.

To do so, we use the geometric extension of $B$. We know that
 $M(\Gamma_1)$ and $M(\Gamma_2)$ are already 
realized as submanifolds of $\mathbb{B}^3$, hence by conjugating
  $M(\Gamma_1)$ by $\hat{f}$ and $M(\Gamma_2)$ by $\hat{g}$,
 we obtain the desired embeddings of $M(\tilde{\Gamma}_1)$ and 
$M(\tilde{\Gamma}_2)$ into $\mathbb{B}^3$.

To complete the proof we note that $H(B)$ is connected, so for maps $f$ and $g$ 
with big distortion, we can use a path on $H(B)$ and extend on small 
distortion changes.

\end{proof}

Now let us show that there are Schottky type extensions for a large set of 
rational maps.

\begin{theorem}There is an open and everywhere dense subset in $Rat_{d}(\C)$ 
which has a geometric extension of Schottky type, of the same degree.
\end{theorem}

\begin{proof} By Lemma \ref{Lem.Luroth.Clebsch}, the union of all the Hurwitz 
spaces of all Blaschke maps of fixed 
degree is open and everywhere dense in $Rat_d(\C)$. Hence
Theorem~\ref{thm.Blasc.Scho} and Lemma~\ref{lem.branch.Hurw} imply this
theorem.
\end{proof}

It follows that structurally stable rational maps have a Schottky type
extension. We believe that any rational map has a geometric extension such that
the respective manifold belongs to the closure of the Schottky space. Since
Hurwitz space of any  branched covering of finite degree of the sphere contains
a rational map, we conjecture that the closure of the Schottky space of given
degree $d$, contains all realizable Hurwitz combinatorics.

\subsection{Extensions of exceptional maps}

Now we discuss the situation when, for a given rational map $R:S_1\rightarrow 
S_2$, the orbifolds
$S_1$ and $S_2$ are equal, so the map $R$ is an orbifold endomorphism. The 
class of maps
$R$ are called exceptional, the reader will find a more detailed discussion of
these maps in \cite{MilLat}.
In particular, the Euler characteristic $\chi(S_1)$ is zero. Hence
$S_1$ is a parabolic orbifold, this only occurs when the map $R$ is
conjugate to either a Tchebichev map, a Latt\'es map or $z\mapsto z^n$.

\begin{theorem}
 Let $G$ be a semigroup of rational maps which are self-coverings of a 
parabolic orbifold $S$ supported on the Riemann sphere. Then
there exist a geometric extension satisfying the following conditions
\begin{itemize}
 \item For every $g\in G$, the extension $\hat{g}$ has the same degree as $g.$
\item Each extension is geometric.
\item The set of extensions $\hat{G}$ is a semigroup, and the extension map is
a homomorphism from $G$ to $\hat{G}$.
\end{itemize}

\end{theorem}

\begin{proof}
The proof exploits the fact that the elements in $G$ have known 
uniformizations.  Consider lattice
$$L_\tau:=\langle z\mapsto z+1, z\mapsto z+\tau: \Im \tau>0\rangle$$ in the 
Latt\'es case, and the lattice 
$$L_0:=\langle z\mapsto z+1, z\mapsto z+\tau: \Im \tau>0\rangle$$ in the case
of $z^n$ and Tchebichev. 
Let $\sigma$ be the involution $z\mapsto -z $. If $\Gamma_\tau$ is 
the group generated by $L_\tau$ and the involution, then $S$ is equivalent to
$\C/\Gamma_\tau$ or $\C/L_0$. Now, we have three groups 
$\Gamma_\tau$, $\Gamma_0$ and the group $L_0$.  In terms of the lattice 
$L_\tau$, $G$ is a semigroup of affine endomorphisms of $L_\tau$. In each case, 
$G$ has a simultaneous Poincar\'e extensions on the orbifolds 
$\mathbb{H}^3/\Gamma$ where $\Gamma$ is one of the three groups mentioned. 
These Poincar\'e extensions  satisfy  the properties of the Theorem.

\end{proof}

As an example we give a detailed description of a three dimensional orbifold
supporting the Poincar\'e extension of integral Latt\'es maps.

In other words, this is the case when $R$ is a holomorphic endormorphism of the
orbifold of type $(\bar{\C},2,2,2,2)$.

 Let us consider the filled torus $T$ in $\C^2$, given in coordinates $(z_1,
z_2)$ by
$|z_1|\leq 1$ and $|z_2|=1$. This space is uniformized by the Poincar\'e
extension of a lattice $L_\tau$ in
$\mathbb{H}^3$ parallel to the plane, for a suitable choice of $\tau$.

Let $I$ be the involution map $$(z_1,z_2)\mapsto (\bar{z}_1,
\bar{z}_2).$$

The map $I$ acts on the filled torus $T$ as an involution. The quotient $T/I$
gives an orbifold $O$ supported in $\bar{\mathbb{B}}^3$ with two ramification 
lines.
Let $\pi:T\rightarrow T/I$ be the orbit projection. The core of $T$ is the unit 
circle $C=\{(z_1,z_2):z_1=0\}.$

Let us consider the family of endomorphisms $\Psi(l,m,k)$ of $T$ given by the
formulae $$(z_1,z_2)\mapsto (z_1^l z_2^k, z_2^m).$$ Where $l,m,k$ are 
integers.  In other words, $\Psi(l,m,k)$
contains all the Poincar\'e extensions of the semigroup of integer 
multiplications on
$T\setminus C$. Then this family commutes
with the involution $I$ and generates a semigroup $J$ of endomorphisms of the
orbifold $O$. 
\begin{corollary}
 Let $\hat{R}$ be an element of $J$, then $\pi(C)$ is an interval which is
invariant with respect to $\hat{R}$ and the restriction of $\hat{R}$ on
$\pi(C)$ is topologically conjugate to a Tchebichev polynomial of degree $k$.
Moreover, there exist a continuous projection $h:\partial O \rightarrow \pi(C)$
so that $h\circ
\hat{R}=\hat{R}\circ h.$
\end{corollary}

\begin{proof} 
 Note that the projection $P:T\rightarrow C$ given by $(z_1,z_2)\rightarrow
(0,z_2)$ commutes with the maps $\Psi(l,m,k)$ and the involution $I$. Moreover,
the restriction of the map $\Psi(l,m,k)$ on $C$ is the power map $z_2\mapsto
z_2^m$. The action of $\hat{R}$ restricted on $\pi(C)$ is topologically
conjugate the Julia set of a Tchebichev map of degree $m$, hence $\pi(C)$ is 
an interval. Since the projection $P$ commutes with
the action of $I$, descends to a projection $h$ as desired.
\end{proof}

One would expect  that Tchebichev polynomials are obtained by pinching
$T$ onto $C$. The projection $P$ defines a foliation $\partial T$ by circles.
Hence, the projection $h$ defines a foliation $F$ on $\partial O$, all leaves
in $F$ are topological circles with the exception of two leaves homeomorphic
to intervals. In other words, one would think that the foliation $F$ shrinks
to a model of the Tchebichev polynomial. So there would be a
deformation of the foliation in the boundary $\partial O$ that produces a
Tchebichev map. The corollary above suggests an argument to construct such
deformation. So, it is natural to ask: is it true that the closure of the space
of quasiconformal deformations of flexible Latt\'es maps in $\bigcup_{d<deg(R)} 
Rat_d(\C)$ contains Tchebichev maps?  Is it true that any point in the boundary 
of flexible Latt\'es map is rigid? Finally, is it true that any point in the
boundary has degree strictly smaller than degree of the given Latt\'es map? 

\subsection{Non-Galois affine extensions}

The main idea of this paper is to  transform a rational map to a M\"obius 
morphism. On the sections above, we discussed Galois coverings, which is the 
uniformizable situation. In this subsection, let us consider non Galois 
coverings which also transform rational dynamics into M\"obius dynamics.

 Simple examples  of non Galois coverings are given by Poincar\'e functions
 associated to repelling cycles of $R.$ These are functions $f$
satisfying the functional equation $f(\lambda z)=R^n \circ f$ for some $n$ and 
some $\lambda$.

Let us suppose that there exist an extension $\hat{f}$ of $f$ in 
$\mathbb{H}^3$, then
there exist a multivalued map defined $$K_f=\hat{f}\circ \lambda z \circ
\hat{f}^{-1}.$$ When $K_f$ is an ordinary map, we have a dynamical 
extension. If $f$ is a Galois covering then we are in the parabolic situation 
described in the previous section.  In general, is not clear when  $K_f$ is a 
map even in the case $\hat{f}$ is a visual extension of a Poincar\'e function 
$f$. In the case $K_f$ is a map, we call $\hat{K_f}$ a \textit{non-Galois} 
extension of $R$.

Let $B$ be a Blaschke map, and $R$ a  quasiconformal deformation of $B$. Now 
we show that for any Poincar\'e function $f$ of $R$ there exists
non Galois extension.

\begin{theorem}
Let $R$ be quasiconformally conjugate  to a Blaschke map. For every Poincar\'e 
function of $R$ there exist a non Galois extension of $R$.

\end{theorem}

\begin{proof}
 Let us first consider a Blaschke map $B$.  Any Poincar\'e function
of $B$ satisfies $$f(\bar{z})=\frac{1}{\overline{f(z)}},$$ hence $f$ maps the 
lower half plane $\mathbb{H}_-^2$ to the unit disk. Now we are in position to 
use a similar argument of the proof of Theorem \ref{th.Blasch.dyn} to define an 
extension $\hat{f}:\mathbb{H}^3\rightarrow \mathbb{B}^3$ as follows, 
first identify $\mathbb{H}^3$ with the ``open book'' coordinates $(z,\phi)$ 
where $z\in \mathbb{H}_-^2$ and $\phi$ in the interval $(0,\pi)$ and put
$$\hat{f}(z,\phi)=\tau_\phi(f(z))$$
where $\tau_\phi$ is the M\"obius rotation of angle $\phi$ in $\mathbb{R}^3$ 
with respect to the unit circle. In this case, from the 
equation satisfied by a Poincar\'e function, we have that $K_f=\hat{B}^n$ for 
some iterate of $\hat{B}$, where $\hat{B}$ is the dynamical extension 
constructed in Theorem \ref{th.Blasch.dyn}. Let $R=\phi\circ
B\circ \phi^{-1}$ then any Poincar\'e function for $R$ belongs to the Hurwitz
space of a suitable Poincar\'e function of $B$. Now we can apply the arguments 
in Lemma \ref{lem.branch.Hurw} to finish the proof. 
\end{proof}

Let us recall that, for every complex affine line $L$ there is a process of
hyperbolization $T:L\rightarrow H(L)$ which associates a hyperbolic manifold
$H(L)$ to $L.$ This hyperbolization process is used in the construction of 
3-hyperbolic Lyubich-Minsky laminations \cite{LM}. By this process there is an 
identification $H(L)\cong \C\times
\mathbb{R}_+$. Given an affine line $L$, let us assume that we have fixed
any such identification.

Let $h_t:L\rightarrow L\times t\subset H(L)$ be the horospherical inclusion, then we have
$$h^{-1}_t(\lambda z)=h_{\lambda t}^{-1}(z).$$

Now let $L_1$ and $L_2$ be complex affine lines. Let $F:L_1\rightarrow L_2$ be
any map, then for any $\lambda>0$ there is a family of 
extensions $\hat{F}_\lambda:H(L_1)\rightarrow H(L_2)$ given in
coordinates by $$\hat{F}_\lambda(x,t)=(F(x),\lambda t).$$
Note that if $F$ is affine then there exists a unique $\lambda_0$ such that
$F_{\lambda_0}$ is the Poincar\'e extension of $F$ in $H(L)$.
Let $q_i:L_i\rightarrow \C$ be maps for 
$i=1,2$. Assume we have a polynomial $P$ and an affine map
$\gamma:L_1\rightarrow L_2$ satisfying 

$$P\circ q_1=q_2\circ \gamma.$$

Then for all $\lambda$, $\omega$ and $\rho$ positive real numbers, we have

$$(q_2)_\omega \circ {\gamma}_\lambda(x,t)=(q_2)_\omega(\gamma(x),\lambda
t)=$$
$$(q_2(\gamma(x)),\omega \lambda t)=(P\circ \pi_1(x),\omega \lambda t)=$$
$$P_\rho(q_1(x),\frac{\omega \lambda}{\rho} t)=P_\rho\circ
(q_1)_{\frac{\omega \lambda}{\rho}}(x,t).$$

Now let $\lambda_0$ be the number such that $\gamma_{\lambda_0}$ is the
Poincar\'e extension of $\gamma$ in $\mathbb{H}^3$ and put $\rho=\lambda_0$
in the formula above, then we have 
$$(q_2)_\omega\circ \gamma_{\lambda_0}=P_{\lambda_0}\circ (q_1)_\omega.$$

Assume that $q_1=q_2=f$ where $f$ is the Poincar\'e function of a fixed
point with multiplier $\lambda_0$, then for every $\omega$ the
map $P_{\lambda_0}$ is a geometric dynamical extension with the same degree.  
In this extension, the orbit of every point in $\mathbb{H}^3$ converges to
infinity. In other words, the Julia set of the extension belongs to 
$\hat{\C}$.
To show that $P_{\lambda_0}$ is geometric let $M_1$ be the complement in
$\mathbb{H}^3$ of all vertical lines based on the $P$-preimages of the
postcritical set, then $M_2$ is the complement in $\mathbb{H}^3$ of all
vertical lines based on the postcritical set and $P_{\lambda_0}:M_1\rightarrow
M_2.$ Then $f_\omega$ endows $M_i$ with incomplete M\"obius structures,
making $P_{\lambda_0}$ geometric. 

Another situation is when $\rho=1$, again let $q_1=q_2=f$ as above. Then 
$P_1$ is a Poincar\'e extension, with the manifolds $M_i$. However, the M\"obius
structures on $M_i$ are different, on $M_1$ is given with $f_\omega$ and on
$M_2$ is given by $f_{\lambda_0 \omega}$. Note that $\rho=1$ gives a
homomorphic extension defined on the semigroup of polynomials.

For the reader familiar with the construction of Lyubich-Minksy \cite{LM}, we
note that natural extension of either $P_{\lambda_0}$ or $P_1$ is equivalent
to the $3$-hyperbolic Lyubich-Minsky lamination.

\section{Product extension}
At least for us, it is very surprising that there is a product 
structure on $\mathbb{H}^3$ which, in a sense, is a  ``conformal natural" 
extension of the
complex product on $\C$. To construct this product, first let us extend the
exponential map $Exp(z)=e^{z}$. We consider the coordinates
$(z,t)$ in $\mathbb{H}^3$. Let $h_\alpha: \mathbb{H}^3\rightarrow
\mathbb{H}^3$ be the translation given by $$(x,y,t)\mapsto (x,y-\alpha,t),$$ 
the map $H_\beta:
\mathbb{H}^3\rightarrow
\mathbb{H}^3$ is the dilation given by $$(x,y,t)\mapsto (\beta x,\beta y,
\beta t),$$ and let
$p:\triangle\rightarrow \mathbb{H}^3$  be the stereographic
projection that maps the unit disk $\triangle$ in the unit semisphere in
$\mathbb{H}^3$. 

Put $$\Phi(x,0,t)=p\circ Exp(x+it)$$ and let $V$ be the vertical semiplane over 
the imaginary line.  
Then $\Phi$ maps $V$ onto the unit semisphere in $\mathbb{H}^3$.
Finally, for
$w=(x,y,t)$  let $$\widehat{Exp}(w)=H_{e^{-2\pi y}}\circ \Phi \circ
h_y^{-1}(w).$$
By construction $\widehat{Exp}$ maps $\mathbb{H}^3$ onto $M$, the complement of 
the $t$-axis in $\mathbb{H}^3$, and is a covering. When $t=0$ the
map $\widehat{Exp}$ coincides with the $Exp$. Also $\widehat{Exp}$ defines a
complete M\"obius structure $\delta$ on $M$. Any M\"obius map that leaves  $M$
invariant is M\"obius in $\delta$. 

Since $Exp$  defines an homomorphism of the additive structure on $\C$ onto the
multiplicative structure of $\C^*$.  Then $\widehat{Exp}$ gives a
multiplication ``$*$'' in $M$, which is the push-forward of
the additive structure on $\mathbb{H}^3.$ Let $a$ and $b$ elements in $M$, and 
let $a_1$ and 
$b_1$ be elements such hat $a=\widehat{Exp}(a_1)$ and $b=\widehat{Exp}(b_1).$ Then 
$$a\star b=\widehat{Exp}(a_1+b_1).$$

\begin{lemma}
 The multiplicative structure in $M$ extends to a multiplicative structure on
$\mathbb{H}^3.$
\end{lemma}
\begin{proof}
Let $\|.\|$ the standard norm in the euclidean space $\mathbb{R}^3$. Then for
every $x$ and $y$ in $\mathbb{H}^3$. We have
$$\|x*y\|=\|x\|\|y\|.$$ Now, let  $x=(0,0,t)$. For any $y\in \mathbb{H}^3$
define $y*\tau=(0,0,
t\|y\|)$.  The restriction of this multiplication to points in the axis $t$
in $\mathbb{H}^3$ coincides with the multiplication on $\mathbb{R}_+$. This
definition continuously extends the multiplication $*$ to $\mathbb{H}^3$.
\end{proof}

The multiplication in $\mathbb{H}^3$ is commutative and associative and  has the
following properties:
\begin{itemize}
\item The multiplication on the boundary is the usual multiplication in $\C$.
\item Let $\lambda$ be a non-zero complex number then, for any $x\in
\mathbb{H}^3$, we have $\lambda*x=H_\lambda(x)$. Where $H_\lambda(x)$ is the
Poincar\'e extension of the map $z\mapsto \lambda z$.
\item The unique unit element is $(1,0,0)$. For any $x\neq 0,\infty$ in
$\bar{\mathbb{H}}^3$ there exist $y$ such that $x*y=y*x=(1,0,0)$ and $y=H(x)$
where $H$ is the Poincar\'e extension of the map $z\mapsto 1/z.$ 
\end{itemize}

Now we can define an extension of rational maps. Let $R$ be a rational map and
can be represented as a product of M\"obius maps: 
$$R(z)=\prod \gamma_i(z)$$ where $\gamma_i\in PSL(2,C)$. Hence for $x\in
\bar{\mathbb{H}}^3$ we have a extension with respect to the maps $\gamma_i$
$$\hat{R}(x)=\prod_* P(\gamma_i)(x)$$ where $P(\gamma_i)$ is the 
Poincar\'e extension of $\gamma_i$ in $\bar{\mathbb{H}}^3$. Since the
multiplication is commutative then the definition does not depend on the order 
of the factors $\gamma_i$. The following proposition follows from the 
definition of the product extension.

\begin{proposition}The product extension has the following properties:
\begin{enumerate}
 \item If $\sigma_i$ is the geodesic that connects the pole with the zero of
$\gamma_i$. Then $\hat{R}(\sigma_i)$ is the $t$-axis. 
\item The extension $R\mapsto \hat{R}$
is visual.
\item Any rational map $R$ has a finite number of decompositions in M\"obius
factors. Hence there are only finitely many product extensions for each rational
map $R.$
\end{enumerate}
\end{proposition}


Another extension from $\C$ to $\mathbb{H}^3$ is  induced with the product, is
given by a monomorphism $\Phi$ from the ring of formal series over the usual
multiplication on the complex plane to the ring of formal series with the $*$
multiplication. The map $\Phi$ is continuous on the subring of polynomials.
However, is not clear whether it is still continuous on the subring of
absolutely convergent series. Note that this extension is not conformally
natural, is not even visual. The map $\Phi$ is not a homomorphism
with respect to composition. The biggest semigroup $S$, where $\Phi$
defines a homomorphism with respect to composition, is the generated by
$\lambda z^n$ for any complex $\lambda$ and $n$ a natural number. Moreover,
$\Phi$ on $S$ is conformally natural, geometric and the same degree. In general
is not clear when product and ring extensions are geometric, numerical
calculations of the ring extension of $z^2+c$, with $c$ real, suggests that
this extension is geometric.

\subsection{Some examples of Poincar\'e extensions of quadratic polynomials}

Here we compute some Poincar\'e extensions. These computations are based in the 
following formula for the exponential map defined on the previous section.
\begin{displaymath}
\widehat{Exp}(x,y,t)=(\frac{2e^{t}cos(y)}{1+e^{2t}}e^{x},\frac{2e^{t}sin(y)}{
1+e^ { 2t } } e^ { x},
\frac{e^{2t}-1}{1+e^{2t}}e^{x}).
\end{displaymath}

We have the following facts:
\begin{itemize}
 \item The map $\widehat{Exp}$ is a Poincar\'e extension of 
the map $e^{ z}$. 
\item Let $T$ be group generated by the translation $z\mapsto z+2\pi 
i$, then  the action of $T$ in $\mathbb{C}$ extends an action in $\mathbb{H}^3$ 
generated by the map $(z,t)\mapsto (z+2\pi i,t)$. 
\item The orbit space 
$\mathbb{H}^3/T$ is homeomorphic to $B_L:=\mathbb{H}^3 \setminus L$ 
where $L$ is the $t$-axis.
\item  There exist a complete hyperbolic M\"obius 
structure on $B_L$ so that $\widehat{Ext}:\mathbb{H}^3 \rightarrow B_L$
 defines a M\"obius universal covering map.
\item   The extension from $Exp$ to 
$\widehat{Exp}$ is conformally natural.

\end{itemize}

Let $H_2$ be the Poincar\'e extension of the M\"obius map $z\mapsto 2z$, hence 
the map $\hat{Q}= \widehat{Exp} \circ H_2 \circ \widehat{Exp}^{-1}:B_L 
\rightarrow B_L$
is a Poincar\'e extension of the map $Q(z)=z^{2}$.

For a circle $S$ in $\partial {\mathbb{H}}^{3}$, let us define the 
\textit{dome} over $S$ as the 2-sphere with equator $S$ intersected with 
${\mathbb{H}}^{3}$ and will be denoted by $Dome(S)$.

Using the equations above, we obtain the equation  

$$\hat{Q}({\lambda}(X,Y,T))={\lambda}^{2}(\frac{X^{2}-Y^{2}}{1+T^{2}}, 
\frac{2XY}{1+T^{2}},\frac{2T}{1+T^{2}})$$

with $\lambda \in \mathbb{R}$ and $(X,Y,T)\in Dome(S^{1})$.

If $x={\lambda}X$, $y={\lambda}Y$ and $t={\lambda}T$, then
$||p||^{2}=x^{2}+y^{2}+t^{2}={\lambda}^{2}$. We have that

\begin{displaymath}
\hat{Q}(x,y,t)=(||p||^{2}\frac{x^{2}-y^{2}}{||p||^{2}+t^{2}},||p||^{2}\frac{2xy}
{ ||p||^{2}+t^{2}},||p||^{2}\frac{2t||p||}{||p||^{2}+t^{2}})
\end{displaymath}

In this case, by the formula above, we have that in fact, $\hat{Q}$ extends to 
the whole  $\mathbb{H}^{3}$  and $\hat{Q}(0,0,t)=(0,0,t^2)$. Moreover, 
for every $w$ in $\bar{\mathbb{H}}^3$ we have $\hat{Q}(w)=w*w$ where $*$ is 
the product defined above. The map $\hat{Q}$ commutes with the reflection with 
respect to the dome.

We have the following invariant foliations for the action of $\hat{Q}$:

\begin{enumerate}

\item  \textbf{Semi-spheres centered at the origin}. Observe that parallel 
planes of the form $(x_0, y, t)$: under  $H_2$ are mapped to themselves $(2x_0, 
2y, 2t)$. The map $\widehat{Exp}$ send this foliation to a foliation of domes 
over circles centered at the origin. Hence this domes is a foliation invariant 
under $\hat{Q}$.

\item {\bf Cones centered at the origin}: Horizontal planes (horocyclic 
foliation of ${\B}^{3}$) of the form $(x, y_0, z)$ are invariant under $H_2$, 
hence their image under $\widehat{Exp}$ also. These are cones centered at $0$.

\item {\bf Onion like foliation}: Planes of the form $(x,y,kx)$, are invariant 
under $H_2$. Its image is the onion- like foliation surrounding the vertical 
axis $(0,0,t)$. This is a book decomposition, here the bind of the book is the 
unit circle in the 
boundary plane. 
\end{enumerate}

Let $T_c(x,y,t)=(x+Re(c),y+Im(c), t)$ be the Poincar\'e extension of the map 
$z\rightarrow z+c$ and $L_c$ be the vertical line over $c$ in $\mathbb{H}^3$ 
and  let $B_{L_c}=\mathbb{H}^3\setminus L_c$. If $Q_c(z)=z^2+c$, then we
have an extension $\hat{Q}$ of $Q_c$  depicted in the following commutative 
diagram 

 \begin{displaymath}
     \xymatrix{\mathbb{H}^3
         \ar[d]_{\widehat{Exp}} \ar[r]^{ H_2} & 
 \ar[d]^{T_c\circ \widehat{Exp}}  \mathbb{H}^3 \\
  B_L \ar[r]^{\hat{Q}_c}  & B_{L_c}.}
 \end{displaymath}

The diagram implies that $\hat{Q}_{c}=T_{c} \circ \hat{Q}$. In this case, the 
line $L$ is also the critical line but has complicated dynamics. Let us define 
$K(\hat{Q}_{c})$:= {\bf spatial filled Julia set} of 
$\hat{Q}_{c}$, the set of $(x,y,t)$ such that $\hat{Q}_{c}^{n}(x,y,t)$ does not 
tends to $\infty$ as $n \rightarrow \infty$. 

Using similar arguments as in the one dimensional case, one can show that 
$K(\hat{Q}_c)$ is always bounded in $\mathbb{H}^3$. Also, for parameters $c$ 
with $|c|$ large enough, the critical line converge to infinity. If $V_0$ denotes the 
semiplane $\{(x,y,t)\in \mathbb{H}^3: y=0\}.$ The section 
$K(\hat{Q}_{c}) \cap V_0$ is a bi-dimensional set that we have 
illustrated in Figure \ref{figura.1} for different values of $c$.  
\begin{figure}[htbp]\label{figura.1}
\begin{center}
\includegraphics[scale=0.5]{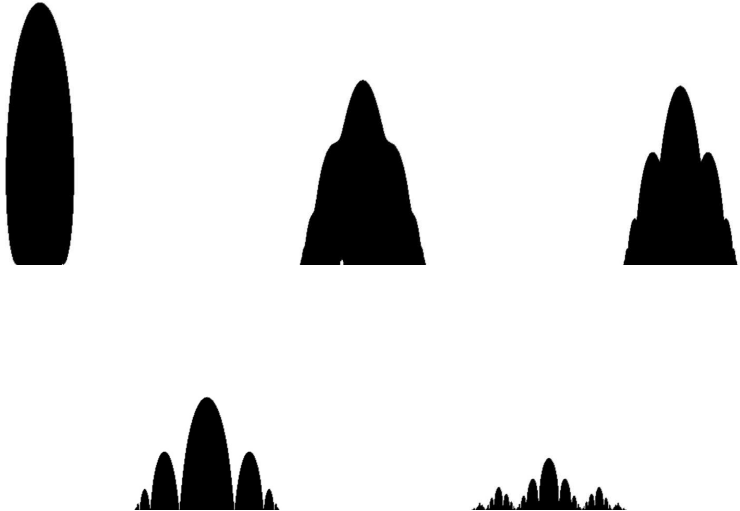}
\caption{The sets $K(\hat{Q}_c) \cap V_0$ for $c=.25,-0.75,-.77,-1, -1.28$ from left to right and top to bottom.}.
\end{center}
\end{figure}

\section{Remarks and conclusions}

As it was mentioned in the introduction, there are some constructions in the 
literature of extensions of rational maps into endomorphisms of $\mathbb{H}^3.$
Most of these extensions are based on the barycentric construction suggested by 
Choquet\'s theorem. Let us briefly describe the barycentric extension. Let 
$\bar{\mathbb{B}}^3$ be the closed unit ball in $\mathbb{R}^3$. Let $\mc{M}$ 
be the space of probability measures on $\partial \mathbb{B}^3$. Then, for 
every $\mu$ in $\mc{M}$ the barycenter
of $\mu$ is the unique point $x$ such that for every functional $L$ on
$\mathbb{R}^3$, the
following equation holds
$$L(x)=\int_{\partial \mathbb{B}^3} L(y)d\mu(y).$$

We define $Bar(\mu)=x$, by Choquet's theorem (See \cite{phelps}, page 48) 
the map $Bar$ sends $\mc{M} $ onto $\bar{\mathbb{B}}^3$. The semigroup
$Rat(\C)$ acts in $\mc{M}$ by push-forward, for every $R$ in $Rat(\C)$ we denote
by $R\mu=R_*(\mu)$ the push-forward of $\mu$ by $R$. For every point $x\in
\mathbb{B}^3$, let $\nu_x$ be the visual measure based on $x$ we define the
barycentric extension of $R$ by $$\hat{R}(x)=Bar(R\nu_x).$$ Then the
barycentric
extension is visual as is proved in \cite{Mc2}. If instead we use the conformal 
barycenter we
obtain a conformally natural extension as is proved in \cite{Carsten}. It is
very difficult to get any geometric information of these extensions. For
instance, it is not clear that these extensions defines a branched covering of
the same degree from $\bar{\mathbb{H}}^3$ to $\bar{\mathbb{H}}^3$.

The following proposition was already mentioned in \cite{LiLiuSufamily}. We
include the proof for completeness.

\begin{proposition}
Let $R$ be a rational map, then all conformally natural extensions of $R$ are
homotopic, with a homotopy that consist of conformally natural extensions of 
$R$.
\end{proposition}

\begin{proof}
 Let $\h{Q}$ and $\h{S}$ be extensions of a rational map $R$, and let $x\in
\mathbb{H}^3$. For $\lambda$ in $[0,1]$, let $E_{\lambda}(x)$ be the point 
along the geodesic from $\h{Q}(x)$ to $\h{S}(x)$, which is at distance $\lambda
d(\h{Q}(x),\h{S}(x))$. Since for every $x\in \partial \mathbb{H}^3$,
$\h{Q}(x)=\h{S}(x)=R(x)$. The map $E_\lambda(x)$ extends to $\partial
\mathbb{H}^3$ as an extension of $R$. If the maps $\h{Q}$ and $\h{S}$ 
are either visual, conformally natural or Poincar\'e, the map $E_\lambda(x)$ 
also holds the same property.
\end{proof}

It follows that if for a given rational map $R$ there are two visual (or
conformally natural) extensions, then there are uncountably many
visual (or conformally natural) extensions. This situation contrasts with the 
product extensions which are only finitely many.

The extension discussed in \cite{martinextension} is uniformly quasiregular
dynamical, has the same degree as the starting map $R$. Moreover, it can be
shown that is geometric. However, for most rational maps these extensions
do not exist \cite{martinextension}.

Another aspect of our geometric construction is about Maskit surgery on the
respective M\"obius manifolds.
A rational map $R:S_1\rightarrow S_2$ is modeled with two groups 
$\Gamma_1<\Gamma_2$.
Let us assume that $\Gamma_2=\langle \gamma_1,...,\gamma_n\rangle$. For $1<k<n$
$G=\langle
\gamma_1,...,\gamma_k\rangle$ and $H=\langle
\gamma_{k+1},...,\gamma_n\rangle$, and consider the intersections 
$G_i=\Gamma_i\cap
G$ and $H_i=\Gamma_i\cap H.$ Then $G_1$ and $H_1$ are subgroups of finite 
index in $G_2$ and $H_2$ respectively. This construction defines two rational 
maps $R_G$ and $R_H$ associated to groups $G_i$ and $H_i$ respectively. Is 
possible to define analogously amalgamated products and HNN-extension.

Now that we have associated a Schottky group to each Blaschke maps, we can use
Maskit combination theorems in order to construct rational maps out of a pair of
Blaschke maps. A theorem of B. Maskit shows that every Schottky group is 
product of
cyclic groups. Let us consider the following example, let $B$ and $B'$ be
two Blaschke products of degree $2$, let $O_1$ and $O_2$ be uniformizations
associated to $B$, and $O_1'$ and $O_2'$ be the uniformization of Schottky type
associated to $B'$ such that the automorphism group of each $O_i$ is Schottky.
Let $\Gamma_1$ and $\Gamma_2$, $\Gamma'_1$ and $\Gamma'_2$ be the associated
Schottky groups, then we construct the product $\tilde{\Gamma}_1=\Gamma_1 *
\Gamma'_1$ and $\tilde{\Gamma}_2=\Gamma_2*\Gamma'_2$. The diagonal action 
of the inclusions gives an inclusion $\tilde{\alpha}:\tilde{\Gamma}_1\rightarrow
\tilde{\Gamma}_2$. Defined on the connected sum of $O_1$ with $O_2$. After 
taking
quotients, we obtain a rational map $R$ which depends only on the combinatorial
data of $B$ and $B'.$ We believe that classical combination constructions in
holomorphic dynamics, such as mating, tuning and surgery are special cases of
the combinations just described.

\bibliographystyle{amsplain} 
\bibliography{workbib}

\providecommand{\bysame}{\leavevmode\hbox to3em{\hrulefill}\thinspace}
\providecommand{\MR}{\relax\ifhmode\unskip\space\fi MR }
\providecommand{\MRhref}[2]{%
  \href{http://www.ams.org/mathscinet-getitem?mr=#1}{#2}
}
\providecommand{\href}[2]{#2}
\begin{thebibliography}{10}

\bibitem{Abikofbarycen}
W.~Abikoff, \emph{{Conformal barycenters and the {D}ouady-{E}arle extension---a
  discrete dynamical approach}}, J. Anal. Math. \textbf{86} (2002), 221--234.

\bibitem{BersSimUniform}
L.~Bers, \emph{{Simultaneous uniformization}}, Bull. Amer. Math. Soc.
  \textbf{66} (1960), 94--97.

\bibitem{bers1970boundaries}
\bysame, \emph{{On boundaries of {T}eichm{\"u}ller spaces and on {K}leinian
  groups: I}}, The Annals of Mathematics \textbf{91} (1970), no.~3, 570--600.

\bibitem{bessonCourGallot}
G.~Besson, G.~Courtois, and S.~Gallot, \emph{{Lemme de {S}chwarz r{\'e}el et
  applications g{\'e}om{\'e}triques}}, Acta Mathematica \textbf{183} (1999),
  no.~2, 145--169.

\bibitem{DouadEarl}
A.~Douady and C.~J. Earle, \emph{{Conformally natural extension of
  homeomorphisms of the circle}}, Acta Math. \textbf{157} (1986), no.~1-2,
  23--48.

\bibitem{KulkarniPinkal}
Ravi~S. Kulkarni and Ulrich Pinkall, \emph{{A canonical metric for {M}{\"o}bius
  structures and its applications}}, Math. Z. \textbf{216} (1994), no.~1,
  89--129. \MR{1273468 (95b:53017)}

\bibitem{LiLiuSufamily}
S.~Li, L.~Liu, and W.~Su, \emph{{A family of conformally natural extensions of
  homeomorphisms of the circle}}, Complex Variables and Elliptic Equations
  \textbf{53} (2008), no.~5, 435--443.

\bibitem{LM}
M.~Lyubich and Y.~Minsky, \emph{{Laminations in holomorphic dynamics}}, J.
  Diff. Geom. \textbf{47} (1997), 17--94.

\bibitem{martinextension}
G.J. Martin, \emph{{Extending rational maps}}, Conform. Geom. Dyn \textbf{8}
  (2004), 158--166.

\bibitem{Maskitbounda}
B.~Maskit, \emph{{On boundaries of {T}eichm{\"u}ller spaces and on {K}leinian
  groups. {II}}}, Ann. of Math. (2) \textbf{91} (1970), 607--639. \MR{0297993
  (45 \#7045)}

\bibitem{McMullenCusps}
C.~McMullen, \emph{{Cusps are dense}}, Ann. of Math. (2) \textbf{133} (1991),
  no.~1, 217--247. \MR{1087348 (91m:30058)}

\bibitem{Mc2}
\bysame, \emph{{Renormalization and 3-manifolds which fiber over the circle}},
  {Annals of Mathematics Studies}, vol. 142, Princeton University Press,
  Princeton, NJ, 1996.

\bibitem{MilLat}
J.~Milnor, \emph{{On {L}att{\`e}s maps}}, {Dynamics on the {R}iemann sphere},
  Eur. Math. Soc., Z{\"u}rich, 2006, pp.~9--43. \MR{2348953 (2009h:37090)}

\bibitem{Carsten}
C.~Petersen, \emph{{Conformally Natural extensions revisited}}, Arxiv Math,\\
  http://arxiv.org/abs/1102.1470 (2011).

\bibitem{phelps}
R.R. Phelps, \emph{{Lectures on {C}hoquet's theorem}}, {Lecture Notes in
  Mathematics}, vol. 1757, Springer-Verlag, Berlin, 2001.

\bibitem{ScGeom}
P.~Scott, \emph{{The geometries of {$3$}-manifolds}}, Bull. London Math. Soc.
  \textbf{15} (1983), no.~5, 401--487.

\bibitem{SK}
R.~Skora, \emph{{Maps between surfaces}}, Trans. Amer. Math. Soc. \textbf{291}
  (1985), no.~2, 669--679. \MR{800257 (87c:57001)}

\end{thebibliography}

\end{document}